\documentclass[11pt,oneside]{amsart}

\usepackage{tikz}
\usepackage{amsthm,amsmath,amssymb}
\usepackage{mathrsfs}
\usepackage{enumerate}
\usepackage{amsfonts}
\usepackage{verbatim}
\usepackage{amsbsy}
\usepackage{amsmath}
\usepackage{amssymb}
\usepackage{MnSymbol}
\usepackage{mathrsfs}
\usepackage{xcolor}
\usepackage[T1]{fontenc}

\newtheorem{theorem}{Theorem}[section]
\newtheorem{lemma}[theorem]{Lemma}
\newtheorem{proposition}[theorem]{Proposition}
\newtheorem{conjecture}[theorem]{Conjecture}

\newtheorem{question}[theorem]{Question}
\newtheorem{corollary}[theorem]{Corollary} 
\theoremstyle{definition}
\newtheorem{definition}[theorem]{Definition}

\theoremstyle{remark}
\newtheorem{remark}[theorem]{Remark}
\newtheorem{example}[theorem]{Example}

\newtheorem{case}{Case}
\newtheorem{subcase}{Case}
\numberwithin{subcase}{case}
\newtheorem{subsubcase}{Case}
\numberwithin{subsubcase}{subcase}

\numberwithin{subsubsubcase}{subsubcase}

\DeclareMathOperator{\rsrank}{rsrank}

\DeclareMathOperator{\Sym}{Sym}
\DeclareMathOperator{\res}{res}


\begin{document}

\title[Ranks of Strictly Minimal Reaction Systems]{Ranks of Strictly Minimal Reaction Systems Induced by Permutations and Cartesian Product}
\author{Wen Chean Teh}
\address{School of Mathematical Sciences\\
	Universiti Sains Malaysia\\
	11800 USM, Malaysia}
\email[Corresponding author]{dasmenteh@usm.my}
\author{Kien Trung Nguyen}
\address{Department of Mathematics\\ Teacher College, Can Tho University\\ Can Tho, Vietnam}
\email{trungkien@ctu.edu.vn}
\author{Chuei Yee Chen}
\address{Department of Mathematics\\ Faculty of Science, Universiti Putra Malaysia\\ Malaysia}
\email{cychen@upm.edu.my}

\begin{abstract}
Reaction system is a computing model inspired by the biochemical interaction taking place within the living cells. Various extended or modified frameworks motivated by biological, physical, or purely mathematically considerations have been proposed and received significant amount of attention, notably in the recent years. This study, however, takes after particular early works that concentrated on the mathematical nature of minimal reaction systems in the context-free basic framework and motivated by a recent result on the sufficiency of strictly minimal reaction systems to simulate every reaction system. This paper focuses on the largest reaction system rank attainable by strictly minimal reaction systems, where the rank pertains to the minimum size of a functionally equivalent reaction system. Precisely, we provide a very detailed study for specific strictly minimal reaction system induced by permutations, up to the quaternary alphabet. Along the way, we obtain a general result about reaction system rank for Cartesian product of functions specified by reaction systems.  	
\end{abstract}

\maketitle{  }

\section{Introduction}

Reaction system \cite{ehrenfeucht2007reaction} is a formal model of natural computing intuitively motivated by the mechanisms of facilitation and inhibition that govern the biochemical iteractions within the living cells.  
Although its original framework is simple, due to its versatile setup and applicability, it has evolved to incorporate various biological  extensions \cite{azimi2015dependency, barbuti2016investigating} and contextual investigations \cite{barbuti2016specialized, bottoni2019reaction, meski2015model}. The room for inquiry is wide open as various novel extended frameworks, for example \cite{bottoni2020networks, ehrenfeucht2017evolving, kleijn2020reaction},
are proposed notably in recent years. 
Purely mathematically motivated studies that exhibit interplay with graph theory can also be found \cite{brijder2012representing, genova2017graph, kreowski2018graph}. It is often compared to the P system and attempts to bridge the two research areas has been initiated \cite{paun2013bridging}. For the latest motivational survey on reaction systems, we refer the reader to \cite{ehrenfeucht2017reaction}.

Our contribution focuses on the mathematical properties of the basic framework.
The early studies in this line of research has revolved mainly around minimal reaction systems \cite{ehrenfeucht2012minimal, salomaa2015applications, salomaa2017minimal} which have a nice algebraic characterization. Despite their simplicity, they are sufficiently rich in the sense that all other functions specified by reaction systems   
over the ternary alphabet can be generated from them under function composition \cite{salomaa2014compositions}, although this is not the case for the quarternary alphabet  \cite{teh2018compositions}. On the other hand, replacing composition by some sense of simulation, Manzoni, Pocas, and Porreca \cite{manzoni2014simple} showed that every reaction system can be simulated if the background set is extended by extra resouces, in particular by reaction systems with each reaction having only a single resource. These simpler reaction systems are called strictly minimal in \cite{teh2020simulation}, where their simulation power is further studied.

Extremal combinatorics for reaction systems was first studied in \cite{dennunzio2015reaction}, where the minimum size after which a reaction system possess a certain property is determined. Independently, reaction system rank was introduced in \cite{teh2017irreducible} as a measure of the complexity of a reaction system  according to the minimum size of a set of reactions functionally equivalent to it. It was found that for every background set $S$, the largest possible reaction system rank $2^{\vert S\vert}$ is effectively attainable, while the largest attainable rank by a reaction system specifying a bijective state transition function is $2^{\vert S\vert}-1$. However, considering that the optimization version of the  set cover problem in \mbox{NP-hard}, finding reaction system rank is computationally infeasible. In fact, the largest reaction system rank attainable by minimal reaction systems was left unsolved and posed as an open problem in \cite{teh2017minimal}. This contribution studies this problem for strictly minimal reaction systems instead, driven by their canonicalness and simple structures.

The remainder of this paper is structured as follows.  Section~2 provides the preliminaries for reaction systems. Subsequently, reaction system rank of Cartesian product of functions specified by general reaction systems is discussed in Section 3. Setting our focus on strictly minimal reaction systems, Section~4 briefly explores irreducibility of these reaction systems. The next three sections address the largest reaction system rank of strictly minimal reaction systems.  Section~5 provides the answer for the ternary alphabet and interprets it as an equivalent refreshing result about some vertex labelling of the cube graph. Meanwhile, the two following sections lead to the answer for the quaternary alphabet by focusing on strictly minimal reaction systems induced by permutations, especially full cycles. The limitation of our approach to cater for higher alphabets, as well as some future directions, is discussed in our conclusion.

\section{Basic Notions of Reaction Systems}\label{1203a}
Throughout this paper, we consider $S$  to be a fixed finite nonempty set. We denote the cardinality of $S$ by $\vert S\vert$, the power set of $S$ by $2^S$, and the symmetric group on $S$  by $\Sym(S)$.

\begin{definition}
A \emph{reaction in $S$} is  a triple
$a=(R_a,I_a, P_a)$, where $R_a$ and $I_a$ are disjoint (possibly empty) subsets of $S$, and $P_a$ is a nonempty subset of $S$. The sets $R_a$, $I_a$, and $P_a$ are the \emph{reactant set}, \emph{inhibitor set}, and \emph{product set}, respectively.
\end{definition}

\begin{definition}
A \emph{reaction system over $S$} is a pair $\mathcal{A}=(S,A)$ where $A$ is a (possibly empty) set of reactions in $S$ and $S$ is called the corresponding \emph{background set}. We say that $\mathcal{A}$ is \emph{nondegenerate} if $R_a$ and $I_a$ are both nonempty for every $a\in A$. 
\end{definition}

\begin{definition}
Suppose $\mathcal{A}=(S,A)$ is a reaction system over $S$. The (state transition) function 
$\res_{\mathcal{A}}\colon 2^S\rightarrow 2^S$ \emph{specified by $\mathcal{A}$} is given by
$$\res_{\mathcal{A}}(X)= \bigcup_{\substack{a\in A\\
R_a\subseteq X, I_a\cap X=\emptyset} } P_a \quad,\quad \text{for all }X\subseteq S.   $$
\end{definition}

When $S$ is understood, we may identify $\mathcal{A}$ with $A$ and write $\res_A$ for $\res_{\mathcal{A}}$. 

Suppose $a$ is a reaction in $S$ and $X\subseteq S$. If $R_a\subseteq X$ and $I_a\cap X=\emptyset$, then we say that $a$ is \emph{enabled by $X$}. The elements in $R_a\cup I_a$ are called the \emph{resources} of $a$.

\begin{definition}
Suppose $\mathcal{A}$ and $\mathcal{B}$ are reaction systems over $S$. We say that $\mathcal{A}$ and $\mathcal{B}$ are \emph{functionally equivalent}  if{f}
$\res_\mathcal{A}= \res_\mathcal{B}$ (that is, $\res_\mathcal{A}(X)= \res_\mathcal{B}(X)$
for all $X\subseteq S$). 
\end{definition}

If $a\in A$, we say that $(R_a,I_a)$ is the \emph{core of $a$}.
It is clear that by taking union of the product sets, reactions with the same core  can be functional equivalently replaced by a single reaction. Henceforth, unless stated otherwise, from Section~\ref{240420a} onwards, we will adopt the following standard convention.
$$\textit{No two distinct reactions in any reaction system have the same core.}$$

\begin{definition}
Suppose $\mathcal{A}=(S,A)$ is a reaction system. We say that $\mathcal{A}$ is \emph{irreducible} if{f} $\res_A\neq \res_B$ for every proper subset $B$ of $A$.
Otherwise, $\mathcal{A}$ is \emph{reducible}.
\end{definition}

We call a function $f\colon 2^S\rightarrow 2^S$ as an \emph{rs function over $S$}, where $rs$ stands for reaction system. Every such function can be canonically specified by some reaction system over $S$. 

\begin{definition}
Suppose $f$ is an rs function over $S$. The \emph{reaction system rank of $f$} (or simply \emph{rs rank}), denoted $\rsrank(f)$, is defined by
$$\rsrank(f)=\min \{\,\vert A\vert :  \mathcal{A}=(S,A) \text{ is a reaction system  such that }\res_{\mathcal{A}}=f    \,\}.$$ 
If $\res_\mathcal{A}=f$ and $\vert A\vert=\rsrank(f)$, we say that $\mathcal{A}=(S,A)$ \emph{witnesses} the rs rank of $f$.
\end{definition}

We may by the $rs$ rank of a reaction system $\mathcal{A}$ refer to the  $rs$ rank of $\res_{ \mathcal{A}}$.
	
The notion of reaction system rank was very recently introduced by Teh and Atanasiu in \cite{teh2017irreducible}. It was shown that the $rs$ rank of $rs$ functions over $S$ is bounded above by  $2^{\vert S\vert}$ and the bound is tight. Futhermore, the largest $rs$ rank for bijective $rs$ functions over $S$ is $2^{\vert S\vert}-1$ and it is effectively attainable.  However, the authors acknowledged that they were unaware at that time that the notion of reducible was called redundant in \cite{dennunzio2015reaction}, where it was study as one of the extremal combinatorics results.

Every reaction system witnessing the $rs$ rank of an $rs$ function is irreducible. However, the cardinality of an irreducible reaction system over $S$ can go beyond $2^{\vert S\vert}$ (see Example~18 in \cite{teh2017minimal}) but any good upper bound remains to be seen. Dropping the previous assumption on reaction systems though, then it was shown that there exists an irreducible reaction system specifying any given $rs$ function $f$ where its size can be arbitrarily chosen between $\rsrank(f)$ and some tight upper bound depending on $f$ \cite{teh2018irreducible}.

\begin{definition}\label{1507d}
	Suppose $\mathcal{A}=(S,A)$ is reaction system.
	We say that  $\mathcal{A}$ is \emph{strictly minimal} if $\vert R_a\cup I_a\vert\leq 1$
	for every reaction $a\in A$.
\end{definition}

In the early studies, reaction systems are 
classified according to the number of resources in its reactions and mostly assumed to be nondegenerate due to naturality.  A simple algebraic characterization for $rs$ functions specified by minimal reaction systems (that is, where $\vert R_a\vert=\vert I_a\vert= 1$ for all reactions $a\in A$) was obtained in \cite{ehrenfeucht2012minimal}. Later, Manzoni et al.~\cite{manzoni2014simple} showed that the class of  reaction systems where every reaction has only a single resource is sufficiently rich to simulate every $rs$ function over $S$ in some sense. This study of simulation was continued in \cite{teh2020simulation}, where the term strictly minimal was first introduced.

\section{Reaction system rank of Cartesian product}

In this section, we study reaction system rank of Cartesian product of $rs$ functions.

\begin{definition}\label{260420e}
Suppose $f$ is an $rs$ function over $S$ and suppose $g$ is an $rs$ function over $T$, where $S\cap T=\emptyset$.
The \emph{Cartesian product of $f$ and $g$}, denoted $f\times g$, is the $rs$ function over $S\cup T$ defined by
$$(f\times g)(X)= f(X\cap S)\cup g(X\cap T), \; \text{ for all } X\subseteq S\cup T.$$ 	
\end{definition}

Definition~\ref{260420e} takes after the usual definition of the Cartesian product of $f$ and $g$ where it is the operation  on $2^S\times 2^T$ that maps $(X,Y)$ to $(f(X), g(Y))$ for all $(X,Y)\in 2^S\times 2^T$. Since $S$ and $T$ are disjoint, the tuple $(X,Y)$ can be identified with the subset $X\cup Y$ of $S\cup T$ while $(f(X), g(Y))$ can be identified with $f(X)\cup g(Y)$. 

\begin{remark}\label{2603b}
	If $A$ is a reaction system over $S$ and $B$ is a reaction system over $T$, then $A\cup B$ is a reaction system over $S\cup T$ and	
	$\res_{A\cup B}= res_A\times \res_B$.	
\end{remark}

For the rest of this section, we assume $S\cap T=\emptyset$. Also, we fix an $rs$ function $f$ over $S$ and an $rs$ function $g$ over $T$. Our aim is to study $\rsrank(f\times g)$ in terms of $\rsrank(f)$ and $\rsrank(g)$. We begin with the following trivial upper bound.

\begin{lemma}\label{2603a}
$\rsrank(f\times g)\leq  \rsrank(f)+\rsrank(g)$.	
\end{lemma}

\begin{proof}
Suppose $A$ is a reaction system over $S$ witnessing the $rs$ rank of $f$ and $B$ is a reaction system over $T$ witnessing the $rs$ rank of $g$. Then by Remark~\ref{2603b}, $A\cup B$ is a reaction system over $S \cup T$ such that
$\res_{A\cup B}= f\times g$. Therefore, $\rsrank(f\times g)\leq \vert A\cup B \vert=\rsrank(f)+\rsrank(g)$.
\end{proof}

Next, we show that equality $\rsrank(f\times g)=  \rsrank(f)+\rsrank(g)$ holds when the corresponding reaction systems are nondegenerate.
For convenience, we do not adopt the convention that distinct reactions in a reaction system have distinct cores.

\begin{lemma}\label{1003a}
Suppose $A$ is a reaction system over $S\cup T$ such that 
$f\times g =\res_{A}$. Then $f$ can be specified by the following reaction system $A'$ over $S$, where
	$$A'= \{\, (R_a, I_a\cap S, P_a\cap S)\mid a\in A, R_a\subseteq S, \text{and } P_a\cap S\neq \emptyset  \,\}.$$	
\end{lemma}

\begin{proof}
Suppose $X\subseteq S$ is arbitrary. We need to show that $f(X)= \res_{A'}(X)$. Suppose $y\in \res_{A'}(X)$. Then 
$(R_a, I_a\cap S, P_a\cap S)$ is enabled by $X$ for some $a\in A$ such that $R_a\subseteq S$ and $y\in P_a\cap S$. It follows that $a$ is enabled by $X$ and thus $y\in \res_A(X)$. Since $\res_{A}(X)= f(X) \cup g(\emptyset)$ and $g(\emptyset) \subseteq T$, it follows that $y\in f(X)$.

Conversely, suppose $y\in f(X)$. Then $y\in f(X) \cup g(\emptyset)=\res_A(X)$. Hence, $a$ is enabled by $X$ for some $a\in A$ such that $y\in P_a$. It follows that the reaction  $(R_a, I_a\cap S, P_a\cap S)$ in $ A'$ is enabled by $X$. Therefore, $y \in \res_{A'}(X)$. 
\end{proof}

\begin{theorem}\label{0903a}
Suppose $f$ is an $rs$ function over $S$ and suppose $g$ is an $rs$ function over $T$, where $S\cap T=\emptyset$.
If $f(\emptyset)= g(\emptyset)=\emptyset$, then $\rsrank(f\times g)= \rsrank(f)+\rsrank(g)$.	
\end{theorem}

\begin{proof}
Suppose $A$ is a reaction system over $S\cup T$ witnessing the $rs$ rank of $f\times g$. By Lemma~\ref{1003a}, $f$ can be specified by the reaction system
	$$A'_S= \{\, (R_a, I_a\cap S, P_a\cap S)\mid a\in A, R_a\subseteq S \text{ and } P_a\cap S\neq \emptyset  \,\}.$$	
Due to symmetry, $g$ can be specified by the reaction system
	$$A'_T= \{\, (R_a, I_a\cap T, P_a\cap T)\mid a\in A, R_a\subseteq T \text{ and } P_a\cap T\neq \emptyset  \,\}.$$	
Since $f(\emptyset)= g(\emptyset)=\emptyset$, the reactant set of every reaction in $A'_S\cup A'_T$ is nonempty. Since $S$ and $T$ are disjoint, it follows that 
$$\{ \, a\in A\mid R_a\subseteq S \text{ and } P_a\cap S\neq \emptyset  \,    \}\cap 
\{ \, a\in A\mid R_a\subseteq T \text{ and } P_a\cap T\neq \emptyset  \,    \}  =\emptyset.$$
Hence, $\vert A\vert \geq \vert A'_S\vert+\vert A'_T\vert$.
Therefore, $\rsrank(f\times g) = \vert A\vert \geq \rsrank(f)+\rsrank(g)$.	
\end{proof}

\begin{corollary}\label{120120a}
Suppose $B$ and $C$  are nondegenerate reaction systems over $S$ and $T$, respectively, where $S\cap T=\emptyset$. Then
$B\cup C$ is a reaction system over $S\cup T$ such that 
	$\rsrank(\res_{B\cup C}) =\rsrank(\res_B)+ \rsrank(\res_C)$.	
\end{corollary}

There is indeed a lower bound for the $rs$ rank of $f\times g$, namely $\rsrank(f\times g)\geq \rsrank(f)+\rsrank(g)-2$. While this lower bound will not be proved until Theorem~\ref{3103a}, we first illustrate an example showing that this lower bound is valid.

\begin{example}
	Let $S=\{1,2,3\}$ and $T=\{4,5,6\}$. We consider the reaction system $A=\{  (\emptyset, \emptyset, \{1\}), (\{1\},\emptyset, \{2\}), (\emptyset, \{1\},\{3\})\}$ over $S$ and also the reaction system $B=\{  (\emptyset, \emptyset, \{4\}), (\{4\},\emptyset, \{5\}), (\emptyset, \{4\},\{6\})\}$ over $T$. Then 	
	$\rsrank(\res_A)=\rsrank(\res_B)=3$.  However, $A\cup B$ is functionally equivalent to
	$$\{(\{1\},\emptyset,\{2,4\}), (\emptyset,\{1\},\{3,4\}), (\{4\},\emptyset, \{1,5\}), (\emptyset,\{4\},\{1,6\})\}$$ and thus
	$\rsrank(\res_{A\cup B})=\rsrank(\res_A)+\rsrank(\res_B)-2$.	
\end{example}

The following lemma mirrors Lemma~\ref{1003a}. It is due to the observation that $f$ is also embedded in $f\times g$ in the following sense:
$$(f\times g)(X\cup T) =f(X)\cup g(T), \; \text{for all } X\subseteq S.$$
The proof is omitted due to similarity.

\begin{lemma}\label{1003b}
	Suppose $A$ is a reaction system over $S\cup T$ such that 
	$f\times g =\res_{A}$. Then $f$ can be specified by the following reaction system $A'$ over $S$, where
		$$A'= \{\, (R_a \cap S, I_a, P_a\cap S)\mid a\in A, I_a\subseteq S, \text{and } P_a\cap S\neq \emptyset   \,\}.$$	
		\end{lemma}



The following is a technical lemma due to some counting argument.

\begin{lemma}\label{0903c}
Suppose  $A$ is a reaction system over $S\cup T$. Let 
\begin{gather*}
A_{S,r}'= \{\, (R_a, I_a\cap S, P_a\cap S)\mid a\in A, R_a\subseteq S, P_a\cap S\neq \emptyset, \text{and } (R_a, I_a\cap S)\neq (\emptyset, \emptyset)  \,\},\\
A_{T,r}'= \{\, (R_a, I_a\cap T, P_a\cap T)\mid a\in A, R_a\subseteq T, P_a\cap T\neq \emptyset, \text{and } (R_a, I_a\cap T)\neq (\emptyset, \emptyset)  \,\},\\
A_r''= \{ \,a\in A \mid R_a\subseteq S \text{ or } R_a\subseteq T   \,\},\\
A_r'''= \{\,a\in A\mid R_a=\emptyset, I_a\cap S\neq \emptyset, \text{and } I_a\cap T\neq \emptyset \, \}.
\end{gather*}
Then $\vert A_r''\vert+\vert A_r'''\vert \geq \vert A_{S,r}'\vert +\vert A_{T,r}'\vert$.
\end{lemma}

\begin{proof}
Let $B_S= \{\, a\in A \mid R_a\subseteq S, P_a\cap S\neq \emptyset, \text{and } (R_a, I_a\cap S)\neq (\emptyset, \emptyset)  \,\}$ and $B_T= \{\, a\in A \mid R_a\subseteq T, P_a\cap T\neq \emptyset, \text{and } (R_a, I_a\cap T)\neq (\emptyset, \emptyset)  \,\}$.
Note that $\vert A'_{S,r}\vert \leq  \vert  B_S \vert$ and $\vert A'_{T,r}\vert \leq  \vert B_T\vert $. 
Furthermore, $B_S\cup B_T \subseteq A''_r$ and  $B_S\cap B_T \subseteq A'''_r$. Hence,
$$\vert A''_r\vert \geq \vert B_S\cup B_T\vert = \vert B_S\vert +\vert B_T\vert - \vert B_S\cap B_T\vert \geq \vert A'_{S,r}\vert +\vert A'_{T,r}\vert- \vert A'''_r\vert.$$
Therefore, $\vert A_r''\vert+\vert A_r'''\vert \geq \vert A_{S,r}'\vert +\vert A_{T,r}'\vert$.
\end{proof}

Analogously, we have the following lemma.
	
\begin{lemma}\label{0903d}
Suppose $A$ is a reaction system over $S\cup T$. Let 
\begin{gather*}
A_{S,i}'= \{\, (R_a\cap S, I_a, P_a\cap S)\mid a\in A, I_a\subseteq S, P_a\cap S\neq \emptyset, \text{and } (R_a\cap S, I_a)\neq (\emptyset, \emptyset)  \,\},\\
A_{T,i}'= \{\, (R_a\cap T, I_a, P_a\cap T)\mid a\in A, I_a\subseteq T, P_a\cap T\neq \emptyset, \text{and } (R_a\cap T, I_a)\neq (\emptyset, \emptyset)  \,\},\\
A_i''= \{ \,a\in A \mid I_a\subseteq S \text{ or } I_a\subseteq T   \,\},\\
A_i'''= \{\,a\in A\mid I_a=\emptyset, R_a\cap S\neq \emptyset, \text{and } R_a\cap T\neq \emptyset \, \}.
\end{gather*}
Then $\vert A_i''\vert+\vert A_i'''\vert \geq \vert A_{S,i}'\vert +\vert A_{T,i}'\vert$.
\end{lemma}

\begin{theorem}\label{3103a}
Suppose $f$ is an $rs$ function over $S$ and suppose $g$ is an $rs$ function over $T$, where $S\cap T=\emptyset$. Then
$$\rsrank(f)+\rsrank(g)-2\leq \rsrank(f\times g)\leq \rsrank(f)+\rsrank(g).$$	
\end{theorem}

\begin{proof}
Suppose $A$ is a reaction system over $S\cup T$ witnessing the $rs$ rank of $f \times g$. By Lemma~\ref{2603a}, it suffices to show that $\vert A\vert \geq \rsrank(f)+\rsrank(g)-2$.
Let $A_{S,r}', A_{T,r}', A_r'', A_r'''$ and 
$A_{S,i}', A_{T,i}', A_i'', A_i'''$ be defined as in Lemmas~\ref{0903c} and \ref{0903d}. 

First, we show that $\vert A_{S,r}'\vert \geq \rsrank{f}-1$.
Let 
$$A_{S,r}= \{\, (R_a, I_a\cap S, P_a\cap S)\mid a\in A, R_a\subseteq S, \text{and } P_a\cap S\neq \emptyset  \,\}$$
as in Lemma~\ref{1003a} and let $P= \bigcup \{ \,  P_a\cap S\mid a\in A \text{ and } (R_a, I_a\cap S) = (\emptyset, \emptyset)\, \}$. Then it can be observed that $A_{S,r}$ is functionally equivalent to $A_{S,r}' \cup \{ (\emptyset, \emptyset, P)  \} $. By Lemma~\ref{1003a}, it follows that $\vert A'_{S,r}  \cup \{ (\emptyset, \emptyset, P)  \}  \vert =\vert A_{S,r}'\vert +1   \geq \rsrank(f)$. Similarly, it can be shown that $\vert A_{T,r}'\vert \geq \rsrank(g)-1$.

Now, by Lemma~\ref{0903c},
$\vert A_r''\vert+\vert A_r'''\vert \geq \vert A_{S,r}'\vert +\vert A_{T,r}'\vert\geq \rsrank(f) +\rsrank(g)-2$.
Similarly, using Lemmas~\ref{1003b} and \ref{0903d}, it can be shown that $\vert A_i''\vert+\vert A_i'''\vert\geq \rsrank(f) +\rsrank(g)-2$.
Assume $\vert A_r'''\vert\geq  \vert A_i'''\vert$ as the other case is similar.
Note that $A_i''$ and  $A_r'''$ are disjoint.
Therefore, $\vert A\vert \geq \vert A_i''\vert+\vert A_r'''\vert
\geq \vert A_i''\vert+\vert A_i'''\vert \geq \rsrank(f) +\rsrank(g)-2$.
\end{proof}

In view of Theorems~\ref{0903a} and \ref{3103a}, we consider more general conditions such that $\rsrank(f\times g)= \rsrank(f)+\rsrank(g)$ holds.

\begin{theorem}\label{1003d}
Suppose both the following conditions hold:
	\begin{enumerate}
		\item No proper $X\subseteq S$ exists such that $X\subseteq f(Y)$ for all $Y\subseteq S$; 
		 \item No proper $X\subseteq T$ exists such that $X\subseteq g(Y)$ for all $Y\subseteq T$.		
	\end{enumerate}
		Then $\rsrank(f\times g)= \rsrank(f)+\rsrank(g)$.	
\end{theorem}

\begin{proof}
Suppose $A$ is a reaction system over $S\cup T$ witnessing the $rs$ rank of $f \times g$.
Let $A_{S,r}', A_{T,r}', A_r'', A_r'''$ and 
$A_{S,i}', A_{T,i}', A_i'', A_i'''$ be  defined as in Lemmas~\ref{0903c} and \ref{0903d}. 
By Lemma~\ref{1003a} and condition $(1)$, it follows that $f=\res_{A'_{S,r}}$ and thus
$\vert A_{S,r}'\vert \geq \rsrank(f)$. Similarly, it can be shown that
$\vert A_{T,r}'\vert\geq \rsrank(g)$. Hence, 
$\vert A_r''\vert+\vert A_r'''\vert \geq \rsrank(f)+\rsrank(g)$ by Lemma~\ref{0903c}.
Analogously, it can be verified that $\vert A_i''\vert+\vert A_i'''\vert\geq \rsrank(f) +\rsrank(g)$. As in the proof of Theorem~\ref{3103a}, we may assume $\vert A_r'''\vert\geq  \vert A_i'''\vert$.
Hence, $\vert A\vert \geq \vert A_i''\vert+\vert A_r'''\vert
\geq \vert A_i''\vert+\vert A_i'''\vert \geq \rsrank(f) +\rsrank(g)$.
\end{proof}

The sufficient condition in Theorem~\ref{1003d} is not a necessary condition. For example, let $B=\{ (\{1 \}, \emptyset, \{1,2\}), (\emptyset, \{1\}, \{2\})\}$ and $C= \{( \{3\},  \emptyset, \{3\})  \}$. 
Then it can be verified that $\rsrank(\res_{B\cup C}) =3=\rsrank(\res_B)+ \rsrank(\res_C)$. However,
$\{2\}\subseteq \res_{B}(Y)$ for all $Y\subseteq \{1,2\}$.

\section{Irreducible strictly minimal reaction systems}\label{240420a}

From now onwards, we focus our study on strictly minimal reaction systems, of which we first begin with their irreducibility in this section. For this purpose, we introduce the following two conditions:
\begin{itemize}
	\item[A1.] Whenever $a=(\{s\}, \emptyset, P_a)$ and $b=(\emptyset,\{s\}, P_b)$ are reactions in $A$ for some $s\in S$, then $P_a\cap P_b=\emptyset$.
	\item[A2.] If $(\emptyset, \emptyset, P)\in A$, then $P\cap  P_a=\emptyset$ for all $a\in A \backslash \{ (\emptyset, \emptyset, P)   \}$.
\end{itemize}
Conditions A1 and A2 are canonical because one can verify that every strictly minimal reaction over $S$ is functionally equivalent to some strictly minimal reaction system $(S,A)$ satisfying these conditions.

\begin{lemma}\label{161219a}
Let $q\in S$. No two distinct strictly minimal reaction systems $(S,A)$ and $(S,B)$ satisfying conditions A1 and A2 such that $P_a=\{q\}$ for all $a\in A\cup B$ are functionally equivalent.
\end{lemma}

\begin{proof}
Suppose $(S,A)$ and $(S,B)$ are strictly minimal reaction systems 
as stated in the lemma and they are functionally equivalent. 
	Let $R_B= \{ \, x\in S\mid  (\{x\}, \emptyset, \{q\} )\in B   \}$ and
	$I_B=\{\, y\in S \mid (\emptyset, \{y\}, \{q\} ) \in B\,\}$. By condition A1, $R_B\cap I_B=\emptyset$.
	If $(\emptyset, \emptyset, \{q\})\in A\backslash B$, 
	then $\res_{\mathcal{A}}(I_B)=\{q\}$ but $\res_{\mathcal{B}}(I_B)=\emptyset$, which gives a contradiction. Hence, by condition A2, either $A=B=\{ (\emptyset, \emptyset, \{q\})  \}$ or $(\emptyset, \emptyset, \{q\})\notin A\cup B$. Suppose the latter holds.

Assume $A\nsubseteq B$.
If $(\{x\}, \emptyset, \{q\})\in A\backslash B$ for some $x\in S$, then $\res_{\mathcal{A}}(\{x\}\cup I_B)=\{q\}$ but $\res_{\mathcal{B}}(\{x\}\cup I_B)=\emptyset$. Otherwise, if $(\emptyset, \{y\}, \{q\})\in A\backslash B$ for some $y\in S$, then $\res_{\mathcal{A}}(I_B)=\{q\}$ but $\res_{\mathcal{B}}(I_B)=\emptyset$. In either case, it gives a contradiction. Hence, $A\subseteq B$. Due to symmetry, $B\subseteq A$ and thus $A=B$.	
\end{proof}

\begin{theorem}
	No two distinct strictly minimal reaction systems satisfying conditions A1 and A2 are functionally equivalent.	
\end{theorem}

\begin{proof}
	Suppose $\mathcal{A}=(S,A)$ and $\mathcal{B}=(S,B)$ are two strictly minimal reaction systems 
	satisfying conditions A1 and A2 are functionally equivalent. For each $q\in S$,
	let $A^q =\{ \, (R_a, I_a, \{q\} )\mid a\in A\text{ and } q\in P_a  \,\}$ and let $B^q$ be defined similarly.
	Then $A^q$ and $B^q$ also satisfy conditions A1 and A2 and they are functionally equivalent. Hence, by Lemma~\ref{161219a}, $A^q= B^q$. Since this is true for all $q\in S$, it follows that $A=B$.	
\end{proof}

It is now obvious that the following corollary holds. 

\begin{corollary}\label{191119a}
	Every strictly minimal reaction system satisfying conditions A1 and A2 is irreducible.	
\end{corollary}


As a further corollary, we address the cardinalities of irreducible strictly minimal reaction systems.

\begin{corollary}	
	Suppose $\vert S\vert \geq 3$. The largest irreducible strictly minimal reaction system over $S$ has cardinality $2\vert S\vert+1$.	
\end{corollary}

\begin{proof}
Let $S=\{1,2,3, \dotsc\}$. Every  strictly minimal reaction system over $S$ has cardinality at most $2\vert S\vert+1$ due to our convention. Consider the strictly minimal reaction system satisfying conditions A1 and A2:
$$A=\{(\emptyset, \emptyset, \{1\})  \}\cup  \{\, (\{s\}, \emptyset, \{2\}) \mid s\in S \,\}     \cup \{\, (\emptyset,\{s\}, \{3\}) \mid s\in S \,\}.$$
 By Corollary~\ref{191119a}, $A$ is irreducible.  
\end{proof}

\section{Reaction system rank of strictly minimal reaction systems}

In this section, we initiate the study on the following question.

\begin{question}\label{2003a}
	What is the largest reaction system rank attainable by an rs function specified by a  strictly minimal reaction systems over $S$? 
\end{question}

\begin{example}
	The set $\{ (\{1\},\emptyset, \{1\}), (\{2\}, \emptyset, \{2\}), (\emptyset, \{1\},\{2\}), (\emptyset, \{2\},\{1\})     \}$ of reactions is functionally equivalent to 
	$\{ (\{1\},\emptyset, \{1\}), (\{2\},\emptyset, \{2\}), (\emptyset,\{1,2\},\{1,2\})  \}$. 
\end{example}

In fact, for the binary case, the answer to Question~\ref{2003a} is three. This can be easily verified as the number of cases is small. For a general background set $S$, the corresponding largest $rs$ rank is bounded by $2\vert S\vert$, as the reaction with core being $(\emptyset, \emptyset)$ can be absorbed by other reactions.
We will see in this section that for the ternary case, the answer is five instead of six but the quaternary case will be resolved only in Section~\ref{250420b}. Before that, we first establish some preliminary results  that allow reduction of the number of cases to be considered.


\begin{remark}\label{200420d}
Suppose $A$ is a reaction system over $S$ and $\sigma$ is a permutation on $S$.
Let $B= \{\,(R_a, I_a, \sigma[P_a] ) \mid a\in A \,\}$. Then $\rsrank(\res_B)= \rsrank(\res_A)$.
\end{remark}


\begin{lemma}\label{260420c}
Suppose  $A$ is a strictly minimal reaction system over $S$ and $q\in S$.
Suppose $(\{q\}, \emptyset, P)\in A$ and $(\emptyset, \{q\}, Q)\in A$ for some $P$ and $Q$. Let 
$$B= A\backslash \{( \{q\}, \emptyset, P), (\emptyset, \{q\}, Q) \} \cup \{( \{q\}, \emptyset, P), (\emptyset, \{q\}, Q) \}.$$ 
Then $\rsrank(\res_B)=\rsrank(\res_A)$.
\end{lemma}

\begin{proof}
First, we make a general observation. Suppose $C$ is any reaction system over $S$. Let
\begin{multline*}
D= \{\, c\in C  \mid  q\notin R_c\cup I_c  \, \}  \cup \{ \, (R_c \backslash \{q\}, I_c \cup \{q\}, P_c)\mid c\in C \text{ and } q\in R_c\,   \}  \\ \cup \{\, (R_c\cup \{q\}, I_c\backslash \{q\}, P_c)\mid c\in C\text{ and } q\in I_c  \,\}.
\end{multline*}
Then it can be verified that for all $ X\subseteq S$, we have
$$\res_D(X)=\begin{cases}
\res_C(X\backslash \{q\}) &\text{ if } q\in X\\
\res_C(X\cup \{q\}) &\text{ if } q\notin X.
\end{cases}
$$
Note that if $C$ is taken to be $A$, then $D$ becomes $B$ and vice versa. 
Subsequently, suppose $C$ is any reaction system that witnesses the $rs$ rank of $\res_A$. It follows that
$D$ as defined above is functionally equivalent to $B$. Therefore,
$\rsrank(\res_B) \leq \rsrank(\res_A)$. Vice versa, $\rsrank(\res_A) \leq \rsrank(\res_B)$.
\end{proof}


For convenience, we introduce some simplification convention. For example, let $S=\{1,2,3\}$.
We would drop the braces and represent the subsets of $S$ as $\emptyset$,  $1$, $2$, $3$, $12$, $13$, $23$, $123$.
Hence, the reaction $(\{1,2\}, \emptyset, \{3\})$ will be simplified to $(12,\emptyset,3)$.
If $f$ is an $rs$ function over $S$ and $f(\{1,3\})= \{2,3\}$, we will write this as $f(13)= 23$. 

\begin{theorem}\label{2903a}
For $\vert S\vert=3$, the largest reaction system rank attainable by an $rs$ function specified by a strictly minimal reaction systems over $S$ is five.
\end{theorem}

\begin{proof}
Let $S=\{1,2,3\}$. Suppose $A$ is a strictly minimal reaction systems of cardinality exactly six containing
no reaction with core being $(\emptyset, \emptyset)$ such that condition A1 is satisfied. It suffices to show that $\rsrank(A)\leq 5$  as the upper bound is attainable by Theorem~\ref{260420b}.  This is because every strictly minimal reaction system $B$ of cardinality seven can be associated to such a strictly minimal reaction system $A$ and a set $P_0$ such that $\res_B(X) = \res_A(X)\cup P_0$ for all $X\subseteq S$. Hence, if $C$ witnesses the $rs$ rank of $\res_A$, then $\res_B$ can be specified by the reaction system $\{ \,( R_c, I_c, P_c\cup P_0) \mid c\in C \, \}$. Let $f=\res_A$.   

\setcounter{case}{0}

\begin{case}
$\vert P_a\vert =2$ for some $a\in A$.
\end{case}

Because of Lemma~\ref{260420c} and Remark~\ref{200420d}, without loss of generality, we may assume $\{( 1,\emptyset, 1), (\emptyset,1, 23)  \}\subseteq A$ (up to relabelling the elements of $S$).
We now construct a set $B$ of reactions with cardinality five that can specify $f$. First of all, put $ (\emptyset,1, 23)\in B$. Since $1\in f(1)$ and $23\subseteq f(\emptyset)$, we just require an additional specific reaction in $B$ to ensure that $\res_B(\emptyset)= f(\emptyset)$ and $\res_B (1)= f(1)$, namely, if $1\in f(\emptyset)$ then that reaction is $(\emptyset, 23, f(1))$ and otherwise, if 
$1\notin f(\emptyset)$ then that reaction is $(1, 23, f(1))$. Similarly, 
since $1\in f(12)$ and $23\subseteq f(2)$, if $1\in f(2)$ then put $(2, 3, f(12))\in B$ and otherwise, if $1\notin f(2)$ then put $(12, 3, f(12))\in B$. Similarly, we take care of the pairs
$( f(3), f(13))$ and $(f(23), f(123))$.

\begin{case}
$\vert P_a\vert =1$ for all $a\in A$.
\end{case}

By Lemma~\ref{260420c} and Remark~\ref{200420d}, without loss of generality, we may assume $A$ is one of the following.

\begin{subcase}
$A= \{ (1, \emptyset,1), (2,\emptyset, 1), (3, \emptyset,1), (\emptyset, 1, 2), (\emptyset,2,2), (\emptyset, 3, 2) \}$.
\end{subcase}

Then $f$ can be specified by $\{ (\emptyset,123, 2), (123,\emptyset, 1), (1,3,12), (2, 1, 12), (3,2,12)\}$.

\begin{subcase}
$A= \{ (1, \emptyset,1), (2,\emptyset, 2), (3, \emptyset,3), (\emptyset, 1, 2), (\emptyset,2,3), (\emptyset, 3, 1)\}$.
\end{subcase}

Then in fact $\rsrank(f)=5$ by Theorem~\ref{260420b}.

\begin{subcase}
$A= \{ (1, \emptyset,1), (2, \emptyset, 1), (3, \emptyset,1), (\emptyset, 1, 2), (\emptyset,2,2), (\emptyset, 3, 3)\}$.
\end{subcase}

Then $f$ can be specified by $\{ (1, \emptyset,1), (2, 1, 12), (3, \emptyset,1), (\emptyset,2,2), (\emptyset, 3, 3)\}$.
\end{proof}

For the rest of this section, we would like to point an interesting connection of Question~\ref{2003a} to the hypercube graphs. Consider a three-dimensional cube with three of the vertices being $(1,0,0)$, $(0,1,0)$, and $(0,0,1)$.
Its vertices and edges form the cube graph. For $S=\{1,2,3\}$, the well known isomorphism between the Hasse Diagram of $\mathcal{P}(S)$ and the cube graph allows us to identify the elements of $\mathcal{P}(S)$  with the vertices of the cube graph as follows:
\begin{align*}
\emptyset \leftrightarrow (0,0,0), & &\{1\} \leftrightarrow (1,0,0), & &\{2\}\leftrightarrow (0,1,0), & &\{3\} \leftrightarrow (0,0,1),\\
\{1,2\} \leftrightarrow (1,1,0), & &\{1,3\}  \leftrightarrow (1,0,1), & & \{2,3\}\leftrightarrow (0,1,1), & &\{1,2,3\} \leftrightarrow (1,1,1).
\end{align*}

We assign elements of $S=\{1,2,3\}$ as labels to the vertices of the cube graph according to the labelling operations defined as follows.

\begin{enumerate}
\item (Vertex) Assign one or more elements of $S$ as labels to a chosen vertex.
\item (Edge) Assign the same element(s) of $S$ as a label(s) to both end vertices of a chosen edge.
\item (Face) Assign the same element(s) of $S$ as a label(s) to all four vertices of a chosen face.
\end{enumerate}

Theorem~\ref{2903a} is in fact equivalent to the following property regarding the above labelling operations  on the cube graph. 

\begin{theorem}
If the vertices of the cube graph is labelled by subjecting every face to a labelling operation (3) independently but simultaneously, then at most five labelling operations is needed (in any combination) to obtain the same labels on the vertices.
\end{theorem}

For the quaternary alphabet, the corresponding graph is called the tesseract graph. 
However, in this case, we will see later that the largest possible reaction system rank of an $rs$ function specified by a strictly minimal reaction systems over $S$ is eight instead of seven, even when the product sets are restricted to singletons. 

\section{Strictly minimal reaction systems induced by permutations}\label{250420a}

In order to answer Question~\ref{2003a}, we focus on strictly minimal reaction systems induced by permutations defined as follows.

\begin{definition}
Suppose $\sigma \in \Sym(S)$. Let $A_\sigma$ denote the set of reactions
$$A_\sigma=\{ \, (\{s\},\emptyset, \{s\} )\mid s\in S\,\} \cup \{\, (\emptyset, \{s\}, \{ \sigma(s)  \}   )\mid s\in S  \,  \}.$$
We call $\mathcal{A}_\sigma=(S, A_\sigma)$ the \emph{strictly minimal reaction system induced by $\sigma$}. 
\end{definition}

Suppose $\sigma\in \Sym(S)$. Since we are interested in the largest possible reaction system rank of the $rs$ function specified by $\mathcal{A}_\sigma$, we may assume $\sigma$ has no fixed points. Consider the disjoint cycle decomposition of $\sigma= \sigma_1\circ \sigma_2 \circ\dotsb \circ \sigma_l$. By our assumption, each cycle $\sigma_i$ has length at least two. 

For each $1\leq i\leq l$, let $S_i$ denote the set of elements that is moved by the cycle $\sigma_i$. Let $\mathcal{B}_{\sigma_i}$ denote the strictly minimal reaction systems over $S_i$ induced by $\sigma_i$. 
Note that no proper $X\subseteq S_i$ exists such that $X\subseteq \res_{ \mathcal{B}_{\sigma_i}} (Y)$ for all $Y\subseteq S_i$. By means of Remark~\ref{2603b} and Theorem~\ref{1003d},
it follows that $\rsrank( \res_{\mathcal{A}_\sigma } ) = \sum_{i=1}^l  \rsrank(\res_{ \mathcal{B}_{\sigma_i}   }   ) $. Therefore, it suffices to study reaction system rank of the function over $S$ specified by $\mathcal{A}_{\sigma}$ where
$\sigma$ is a full cycle of length $\vert S\vert$.

From now onwards, due to isomorphism, we can further restrict our attention to  $\mathcal{A}_{(1 \, 2\, 3\, \dotsm\, n)}$ and it is understood that $S=\{1,2,\dotsc, n  \} $.
One can also easily verify that the $rs$ rank of the function specified by $\mathcal{A}_{(12)}$ is three.
For the ternary alphabet, we will show in this section that the corresponding $rs$ rank of is five. Before that, we introduce a way to filter out reactions that are maximal in some sense with respect to a given $rs$ function.  This provide us a more efficient way to study reaction system rank as the list of pertinent reactions is shorter.

\begin{definition}\label{070420a}
Suppose $f$ is an $rs$ function over $S$ and $a$ is a reaction over $S$. 
We say that $a$ is \emph{$f$-compatible} if $P_a\subseteq f(X)$ whenever $X$ is enabled by $a$.
  We say that an $f$-compatible $a$ is \emph{maximally $f$-compatible} if no distinct $f$-compatible $b$ exists such that $P_a\subseteq P_b$ and $X$ is enabled by $b$ whenever $X$ is enabled by $a$.
\end{definition}

The following proposition explains the rationale of Definition~\ref{070420a}.

\begin{proposition}\label{290420b}
Suppose $f$ is an $rs$ function over $S$. Then there is a reaction system witnessing the $rs$ rank of $f$ consisting of maximally $f$-compatible reactions.
\end{proposition}

\begin{proof}
Suppose $A$ is a reaction system witnessing the $rs$ rank of $f$. By definition, every reaction in $A$ is $f$-compatible. Suppose some $a\in A$ is not maximally $f$-compatible. Then 
$P_a\subseteq P_b$ and $X$ is enabled by $b$ whenever $X$ is enabled by $a$ for some maximally $f$-compatible reaction $b$. Hence, we can just consider $(A\backslash \{a\}) \cup \{b\}$.
\end{proof}

\begin{figure*}
\centering
\begin{tikzpicture}[scale=0.6]
     \node (123) at (0,3) {$123$};
   \node (12) at (-2,1) {$12$};
    \node (13) at (0,1) {$13$};
 \node (23) at (2,1) {$23$};
 \node (1) at (-2,-1) {$13$};
    \node (2) at (0,-1) {$12$};
 \node (3) at (2,-1) {$23$};
     \node (empty) at (0,-3) {$123$};
   \draw (123) -- (12) -- (2) -- (23) -- (123) -- (13) -- (1) -- (12);
\draw (empty) -- (1) -- (13) -- (3) -- (23) -- (2) -- (empty) -- (3);
\end{tikzpicture}
\caption{The $rs$ function $f$ specified by $\mathcal{A}_{(1\,2\,3)}$ represented using the Hasse diagram for $\mathcal{P}(\{1,2,3\})$.} \label{270420d}    
\end{figure*}
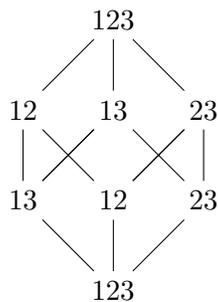

We again adopt our convention outlined before Theorem~\ref{2903a} for simplicity.
In Figure~\ref{270420d}, the underlying framework corresponds to the Hasse diagram for $\mathcal{P}(\{1,2,3\})$, where two underlying subsets (not explicitly shown) of $\{1,2,3\}$ are joined by an edge if and only if they differ by a single element.
At each vertex, the label is the image of the corresponding set under $f$. For example,
the second layer from the bottom shows that $f(1)=13$, $f(2)= 12$ and $f(3)=23$ from left to right.

\begin{theorem}\label{260420b}
Let $f$ be the $rs$ function specified by $\mathcal{A}_{(1\,2\,3)}$. Then $\rsrank(f)=5$. 
\end{theorem}

\begin{proof}
It is clear that  $\rsrank(f)\leq 5$ because $f$ can be specified by
$$\{ (123, \emptyset, 123), (\emptyset, 123, 123),    (1,2,13), (2,3,12), (3,1,23)      \}. $$
The maximally $f$-compatible reactions are:
\begin{gather*}
(1, \emptyset, 1), (2, \emptyset, 2), (3, \emptyset, 3), (\emptyset, 1, 2), (\emptyset, 2,3), (\emptyset, 3, 1)\\
(12, \emptyset, 12), (13, \emptyset, 13), (23, \emptyset, 23), (\emptyset, 12, 23), (\emptyset, 13,12), (\emptyset, 23, 13)\\
(1,2,13), (2,3,12), (3,1,23), (123, \emptyset, 123), (\emptyset, 123,123).
\end{gather*}

For each maximally $f$-compatible reaction $a$ except $(123, \emptyset, 123)$ and $(\emptyset, 123,123) $, we have 
$\vert \{\, (X,s)\mid a \text{ is enabled by } X \text{ and } s\in P_a  \,\}\vert = 4$.
Since $\vert \{\, (X,s)\mid X\subseteq S \text{ and } s\in f(X)  \,\}\vert =18$, it follows that  
$f$ cannot be specified by a set of maximally $f$-compatible reactions of size less than five. 
\end{proof}

\section{Reaction system rank of the function specified by $\mathcal{A}_{(1234)}$ }\label{250420b}

In this section, we fix our attention to the set $S=\{1,2,3,4\}$ with the objective of showing that the $rs$ rank of the function $f$ specified by $\mathcal{A}_{(1\,2\,3\,4)}$ is eight (Theorem~\ref{200420b}). The proof of this result employs case-by-case exhaustive analysis. Our convention outlined before Theorem~\ref{2903a} will be heavily used to simplify our notations.
To reiterate, $\mathcal{A}_{(1\,2\,3\,4)}$ is the strictly minimal reaction system induced by the cycle $(1\;2\;3\;4) $, that is, 
$$A_{(1\,2\,3\,4)   }= \{ (1, \emptyset,1), (2,\emptyset, 2), (3, \emptyset,3), (4, \emptyset, 4), (\emptyset, 1, 2), (\emptyset,2,3), (\emptyset, 3, 4), (\emptyset, 4,1 )\}.$$

Let $f$ denote the $rs$ function specified by $\mathcal{A}_{(1\,2\,3\,4)}$. The underlying framework  as shown in Figure~\ref{270420c} corresponds to the Hasse diagram for $\mathcal{P}(\{1,2,3,4\})$. 
Note that beginning from left to right, the middle layer of the figure shows that $f(12)=124$, $f(13)= 13$, $f(14)= 134$, $f(23)=123$, $f(24)=24$, and $f(34)=234$ from left to right.

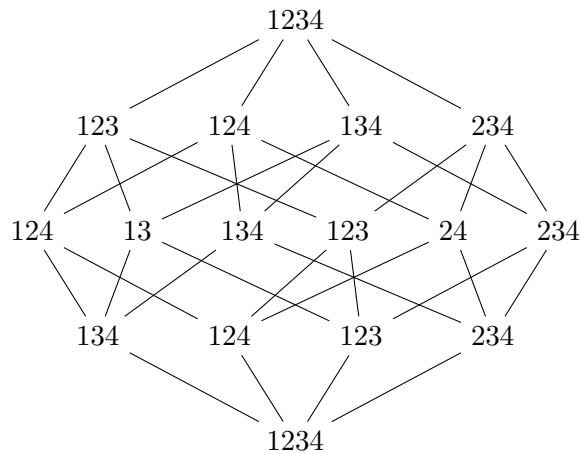
\begin{figure*}
\centering
\begin{tikzpicture}[scale=0.7]
  \node (1234) at (0,4) {$1234$};
  \node (123) at (-3.75,2) {$123$};
  \node (124) at (-1.25,2) {$124$};
  \node (134) at (1.25,2) {$134$};
  \node (234) at (3.75,2) {$234$};
 \node (12) at (-5,0) {$124$};
   \node (13) at (-3,0) {$13$};
    \node (14) at (-1,0) {$134$};
    \node (23) at (1,0) {$123$};
    \node (24) at (3,0) {$24$};
    \node (34) at (5,0) {$234$};
    \node (1) at (-3.75,-2) {$134$};
   \node (2) at (-1.25,-2) {$124$};
    \node (3) at (1.25,-2) {$123$};
 \node (4) at (3.75,-2) {$234$};
     \node (empty) at (0,-4) {$1234$};
   \draw (1234) -- (123) -- (12) -- (124) -- (1234) -- (234) -- (34) -- (134) -- (1234);
\draw (empty) -- (1) -- (12) -- (2) -- (empty) -- (4) -- (34) -- (3) -- (empty);
\draw (123) -- (13) -- (1) -- (14) -- (124) -- (24) -- (2) -- (23) -- (123);
\draw (234) -- (24) -- (4) -- (14) -- (134) -- (13) -- (3) -- (23) -- (234);
\end{tikzpicture}
\caption{The $rs$ function specified by $\mathcal{A}_{(1\,2\,3\,4)}$ represented using the Hasse diagram for $\mathcal{P}(\{1,2,3,4\})$.} \label{270420c}
\end{figure*}

Next, we list down all the maximally $f$-compatible reactions, of which we shall separate into three groups. The first group consists of those where the inhibitor set is empty:
\begin{gather*}
(1, \emptyset,1), (2,\emptyset, 2), (3, \emptyset,3), (4, \emptyset, 4),\\
(12, \emptyset,12), (13,\emptyset, 13), (14, \emptyset,14), (23, \emptyset, 23),
(24, \emptyset,24), (34,\emptyset,34),\\
(123, \emptyset,123), (124,\emptyset, 124), (134, \emptyset,134), (234, \emptyset, 234),\\
(1234, \emptyset, 1234).
\end{gather*}

The second group consists of those where the reactant set is empty:
\begin{gather*}
(\emptyset,1,2), (\emptyset, 2,3), (\emptyset,3,4), (\emptyset, 4,1),\\
(\emptyset,12,23), (\emptyset, 13,24), (\emptyset,14,12), (\emptyset, 23,34),
(\emptyset,24, 13), (\emptyset,34, 14),\\
(\emptyset,123, 234), (\emptyset, 124, 123), (\emptyset,134, 124), (\emptyset, 234, 134),\\
(\emptyset,1234, 1234).
\end{gather*}

The third group consists of those that are nondegenerate, that is, where both the reactant set and the inhibitor set are non-empty:
\begin{gather*}
(1,2, 13), (1,3,14), (2,3,24), (2,4,12),\\
(3,1,23), (3,4,13), (4,1,24), (4,2,34),\\
(12,3,124), (14,2,134), (23,4,123), (34,1, 234),\\
(1,23,134), (2,34,124), (3,14,123), (4,12,234).
\end{gather*}

The following two lemmas will be repeatedly used in the proof of Theorem~\ref{200420b} to reduce the number of cases.

\begin{lemma}\label{020520a}
Let $\sigma=(1 \;2\;3\;4)$. 
Suppose $A$ is any reaction system over $S$ that is functionally equivalent to  $\mathcal{A}_\sigma$. Then  the reaction system $B= \{\, (\sigma^i[R_a], \sigma^i[I_a], \sigma^{i}[P_a]  )\mid a\in A \, \}$ is also functionally equivalent to  $\mathcal{A}_\sigma$ for each $i=1,2,3$.
\end{lemma}

\begin{proof}
Fix $i=1,2,3$. Note that $B$ is isomorphic to $A$ in the sense that the elements of $S$ are relabelled by the permutation $\sigma^i$. Hence, $B$ is functionally equivalent to $\{\, (\sigma^i[R_a], \sigma^i[I_a], \sigma^{i}[P_a]  )\mid a\in A_\sigma \, \}$, the latter being equal to $A_\sigma$.
\end{proof}

\begin{lemma}\label{180420b}
Let $\sigma=(1 \;2\;3\;4)$. 
If $A$ is any reaction system over $S$ that is functionally equivalent to  $\mathcal{A}_\sigma$, then so is the reaction system $B= \{\, (I_a, R_a, \sigma^{-1}[P_a]  )\mid a\in A \, \}$.
\end{lemma}

\begin{proof}
First, we make a general observation. Suppose $A$ is a general reaction system over $S$ and
$B= \{\, (I_a, R_a, \sigma^{-1}[P_a]  )\mid a\in A \, \}$.
Let $C= \{ (I_a, R_a, P_a  )\mid a\in A  \}$.
Fix an arbitrary $X\subseteq S$.
 Then it is easy to see that
$\res_C(X)= \res_A(S\backslash X)$. Also, note that $\res_B(X)= \sigma^{-1}[\res_C(X)]$. 
Therefore, $\res_B(X)= \sigma^{-1}[ \res_A(S\backslash X)     ]$. 

Note that $\{ (I_a, R_a, \sigma^{-1}[P_a]  )\mid a\in A_\sigma \} = A_\sigma$
and so from our observation above, $\res_{A_\sigma}(X)= \sigma^{-1}[ \res_{A_\sigma}(S\backslash X) ]$ for all $X\subseteq S$.
Hence, suppose now that $A$ is any reaction system functionally equivalent to $A_\sigma$ and $B$ is as above. Then for all $X\subseteq S$, we have 
$\res_B(X)= \sigma^{-1}[ \res_A(S\backslash X) ] = \sigma^{-1}[ \res_{A_\sigma}(S\backslash X)     ]
= \res_{A_\sigma}(X)$ as required.
\end{proof}

Before we proceed to our main theorem, we further lay down some terminology to simplify the proof.  We say that a reaction $a$ is \emph{$i$-resourced} if $\vert R_a\cup I_a\vert=i$. For example, a $3$-resourced reaction in $A_2$ is a member belonging to the third row of the second group of maximally $f$-compatible reactions. When we say that a reaction $a$ \emph{accounts for $f(X)$}, it means that $a$ is enabled by $X$ and thus some (or all) of the elements of $f(X)$ are produced/accounted by $a$.   
In most of our case analysis, reactions of $A$ would be one by one supposed or identified until a contradiction is reached. When we say that a reaction $a$ \emph{exhausts $f(X)$}, it means that $a$ is enabled by $X$ and its product set contains the remaining elements of $f(X)$ not accounted by any of the reaction known up to that point.

\begin{theorem}\label{200420b}
Let $f$ be the $rs$ function specified by $\mathcal{A}_{(1234)}$. Then $\rsrank(f)=8$. 
\end{theorem}

\begin{proof}
We argue by contradiction. Assume $f$ can be specified by some reaction system $A$  over $S$ with cardinality seven. By Proposition~\ref{290420b}, we can assume each reaction in $A$ is maximally $f$-compatible.
We write $A$ as a disjoint union of three subsets, in particular $A= A_1\cup A_2\cup A_3$, where 
$$
A_1 ={} \{\, a\in A \mid I_a= \emptyset \, \},\quad
A_2 ={} \{\, a\in A \mid R_a= \emptyset \, \},\quad
A_3 ={} \{\, a\in A \mid R_a\neq \emptyset \text{ and } I_a\neq \emptyset \, \}.
$$

Suppose $a$ is a reaction over $S$. Let
$$\mathcal{H}_a= \{\,(X,s)\mid   \emptyset \neq X\subsetneq S,\, a \text{ is enabled by } X, \text{ and } s\in P_a      \,\}.$$
Also, if $B$ is a set of reactions over $S$, we let $\mathcal{H}_B=\bigcup_{a\in B} \mathcal{H}_a$.
Note that $\vert \mathcal{H}_a \vert \leq 8$  for every maximally $f$-compatible reaction $a$.
Furthermore, observe that 
$$\vert \mathcal{H}_{A}\vert = \vert \{\, (X,s)\mid \emptyset \neq X\subsetneq S \text{ and } s\in f(X)  \,\}\vert =40.\qquad (\star)$$

\setcounter{case}{0}

\begin{case}
$\min \{ \vert A_1\vert, \vert A_2\vert\}=1$. Without loss of generality, Lemma~\ref{180420b} asserts that
$\vert A_2\vert=1$.
\end{case}

Note that then $A_2= \{ (\emptyset, 1234, 1234) \}$.

\begin{subcase}
$A_1= \{ (1234, \emptyset, 1234) \}$.
\end{subcase}
From $(\star)$ above, it follows that $\{\, (X,s)\mid \emptyset \neq X\subsetneq S \text{ and } s\in f(X)  \,\}$ is a disjoint union of 
all five $\mathcal{H}_a$ for $a\in A_3$. 
This implies that $a$ is $2$-resourced for every $a\in A_3$. However, $A_3$ then cannot exhaust
all $f(X)$ for $X=1,2,3,4$, which gives a contradiction.

\begin{subcase}
$\vert A_1\vert \geq 2$.
\end{subcase}

Note that some $a_X\in A_3$ must account for $f(X)$ for each $X=1,2,3,4$. Also, the $a_X$'s must be distinct. 
It follows that $\vert A_3\vert =4$ and $\vert A_1\vert =2$. 
Hence, at most one reaction in $A_1$ is $1$-resourced and thus at least three reactions in $A_3$ is $3$-resourced. 
Also, in any possible combination, we must have $\vert\mathcal{H}_{A_1}\vert \leq 12$.
Hence, $\vert \mathcal{H}_{A}\vert\leq \vert \mathcal{H}_{A_1}\vert+ \vert \mathcal{H}_{A_3}\vert
\leq 12 + 6 \cdot 3 + 8 =38 <40$, which  contradicts $(\star)$.

\begin{case}
$\min \{ \vert A_1\vert, \vert A_2\vert\}=2$. Without loss of generality, Lemma~\ref{180420b} asserts that
$\vert A_2\vert=2$.
\end{case}

We may assume that $(\emptyset, 1234, 1234)\notin A_2$ or else we can replace it by another maximally $f$-compatible reaction from the second group and $f= \res_A$ still holds.

\begin{subcase}\label{090420a}
Both reactions in $A_2$ are $2$-resourced.
\end{subcase}

\begin{subsubcase}
$A_2=\{  (\emptyset, 24, 13), (\emptyset, 13, 24) \}$.
\end{subsubcase}

Given $A$, note that $A_1$ cannot exhaust $f(X)$ for any $X=1,2,3,4$. 
Hence, it follows that $\vert A_3\vert = 4$ and thus $\vert A_1\vert=1 $, which contradicts our case assumption that $\vert A_1\vert\geq 2$.

\begin{subsubcase}\label{080420c}
$A_2=\{  (\emptyset, 34, 14), (\emptyset, 12, 23) \}$ or $A_2=\{ (\emptyset, 14, 12), (\emptyset, 23, 34) \}$ .
\end{subsubcase}

We may suppose $A_2=\{  (\emptyset, 34, 14), (\emptyset, 12, 23) \}$ due to Lemma~\ref{020520a}.
At least one of $(2, \emptyset, 2)$ or $(4, \emptyset, 4)$ is in $A_1$ for otherwise, since there are at most three reactions in $A_3$, some of the $f(X)$ for $X=1,2,3,4$ cannot be exhausted.
Suppose both $(2, \emptyset, 2)$ and $(4, \emptyset, 4)$ are in $A_1$.
Then we must have $\vert A_3\vert =2$ and   $\vert A_1\vert =3$. It follows that  the third reaction in $A_1$ can be assumed to be $(13, \emptyset, 13)$. However, given what we know about $A_1$ and $A_2$,
the reactions in $A_3$ together cannot exhaust $f(1)$, $f(3)$, $f(124)$, and $f(234)$ simultaneously, which gives a contradiction. Hence, we may now suppose $(2, \emptyset,2)\in A_1$ but $(4, \emptyset, 4)\notin A_1$ as the  case $(4, \emptyset,4)\in A_1$ but $(2, \emptyset, 2)\notin A_1$   is similar. Then it must be the case that $\vert A_3\vert =3$ and  $\vert A_1\vert =2$, thus
$A_1= \{  (2, \emptyset, 2), (134, \emptyset, 134)  \}$. The three reactions in $A_3$ together must exhaust 
$f(1)$, $f(3)$, $f(4)$, $f(123)$, $f(124)$, and $f(234)$ and thus each of them is $2$-resourced. However, none of them can simultaneously exhaust $f(124)$ and either $f(1)$, $f(3)$, or $f(4)$.

\begin{subcase}\label{090420e}
The two reactions in $A_2$ are $1$-resourced and $3$-resourced, respectively. 
\end{subcase}

We may suppose $A_2=\{ (\emptyset, 4,1 ), (\emptyset, 123,234)  \}$ due to Lemma~\ref{020520a}.
Note that $A_1\cup A_2$ cannot exhaust any of $f(1)$, $f(2)$, and $f(3)$ regardless of $A_1$. It follows thats $\vert A_3\vert=3$ and $\vert A_1\vert =2$.

\begin{subsubcase}\label{100420a}
Either both reactions in $A_1$ are $2$-resourced or they are $1$-resourced and $3$-resourced respectively.
\end{subsubcase}

In this case, regardless of $A_1$, at least three among $f(X)$ for $X= 123,124,134,234$ are not exhausted by $A_1\cup A_2$. Furthemore, none of $f(1)$, $f(2)$, and $f(3)$ are exhausted by $A_1\cup A_2$. Hence, any reaction from $A_3$ that exhausts $f(1)$ has to exhaust one 
among $f(X)$ for $X= 123,124,134,234$. However, none of the maximally $f$-compatible reactions from the third group can achieve that. 

\begin{subsubcase}\label{100420b}
Either both reactions in $A_1$ are $3$-resourced or they are $2$-resourced and $3$-resourced respectively. 
\end{subsubcase}

In the latter case, if the $2$-resourced reaction in $A_1$ is not $ (23, \emptyset, 23)$, then
at least three among $f(X)$ for $X= 123,124,134,234$ are not exhausted by $A_1\cup A_2$ and  similar argument as in Case~\ref{100420a} works.
Otherwise, for all the other cases, there is some $X^*\in \{124, 134, 234\}$ such that   
none of the elements in $f(X^*)$ have been accounted. 
Since each reaction in $A_3$ must exhaust one of $f(1)$, $f(2)$, and $f(3)$, it
forces $(1,23,134)\in A_3$. Hence, another reaction in $A_3$ must simultaneously exhaust $f(X^*)$ and either $f(2)$ or $f(3)$. However, this is impossible.

\begin{subcase}\label{100420f}
The two reactions in $A_2$ are $2$-resourced and $3$-resourced respectively.
\end{subcase}

We may suppose the $3$-resourced reaction in $A_2$ is $(\emptyset, 123, 234)$ due to Lemma~\ref{020520a}.

\begin{subsubcase}
$\vert A_1\vert=2$.
\end{subsubcase}

Employing Lemma~\ref{180420b}, we may assume that
the two reactions in $A_1$ either are both \mbox{$3$-resourced} or are $2$-resourced and \mbox{$3$-resourced}, respectively. Otherwise, we are reduced to Case~\ref{090420a} or Case~\ref{090420e}. We may further assume $A_2= \{ (\emptyset, 34, 14), (\emptyset, 123, 234) \}$ as the other two cases are similar. The rest of the argument is similar to that of Case~\ref{100420b}. There is some $X^*\in \{123, 124, 134, 234\}$ such that   
none of the elements in $f(X^*)$ have been accounted by $A_1\cup A_2$.
Since each reaction in $A_3$ must exhaust one of $f(1)$, $f(2)$, and $f(3)$, it forces $(3,14,123)\in A_3$. 
Hence, another reaction in $A_3$ must simultaneously exhaust $f(X^*)$  
  and either $f(1)$ or $f(2)$. However, this is impossible.

\begin{subsubcase}
$\vert A_1\vert\geq 3$.
\end{subsubcase}

Regardless of $A_1$, if $A_2= \{ (\emptyset, 24, 13), (\emptyset, 123, 234) \}$, then the reactions (at most two) in $A_3$  cannot exhaust $f(1)$, $f(2)$, and $f(3)$ simultaneously. Now, consider the case $A_2= \{ (\emptyset, 34, 14), (\emptyset, 123, 234) \}$. We must have $(2,\emptyset, 2 )\in A_1$.
Also, to exhaust $f(3)$, we must have $(3, 14, 123)\in A_3$, regardless of whether $(3, \emptyset,3)\in A_1$ or not. However, it can be verified that the remaining reactions in $A_1$ and $A_3$ cannot possibly exhaust $f(1)$, $f(123)$, $f(124)$, $f(134)$, and $f(234)$.   The case $A_2= \{ (\emptyset, 14, 12), (\emptyset, 123, 234) \}$ is similar.

\begin{subcase}
Both reactions in $A_2$ are $3$-resourced.
\end{subcase}

\begin{subsubcase}
$\vert A_1\vert=2$.
\end{subsubcase}

Due to Lemma~\ref{180420b}, we may assume that both reactions in $A_1$ are $3$-resourced. Otherwise, we are reduced to Case~\ref{090420a}, Case~\ref{090420e}, or Case~\ref{100420f}. However, then
$$\vert \mathcal{H}_A\vert \leq  \vert \mathcal{H}_{A_1}\vert +  \vert \mathcal{H}_{A_2}\vert + \vert \mathcal{H}_{A_3}\vert  \leq 3\cdot 2 + 3\cdot 2 + 8 \cdot 3 =36,$$
which contradicts $(\star)$.

\begin{subsubcase}
$\vert A_1\vert\geq 3$.
\end{subsubcase}

We may suppose $A_2=\{(\emptyset, 234, 134), (\emptyset, 134, 124)\}$ as the other cases are similar. 
Regardless of $A_1$, to exhaust $f(3)$, we must have $(3,14,123)\in A_3$. Similarly, to exhaust $f(4)$, we must have $(4,12,234)\in A_3$. Therefore, the three reactions in $A_1$ have to completely account for all $f(X)$ for $X=123, 124, 134, 234$. However, this is impossible.

\begin{case}
$\min \{ \vert A_1\vert, \vert A_2\vert\}=3$. Again, without loss of generality,
$\vert A_2\vert=3$.
\end{case}

\begin{subcase}
$\vert A_1\vert=4$.
\end{subcase}

Note that element $4$ from $f(12)$,  $3$ from  $f(14)$, $1$ from $f(23)$, and $2$ from $f(34)$ 
cannot be accounted by any maximally $f$-compatible reaction from the first group. 
Hence, regardless of $A_1$, no reaction from $A_2$ can simultaneously
exhaust two of the $f(X)$ for $X=12, 14, 23, 34$ and thus $A_2$ cannot exhaust all four of them.

\begin{subcase}
$\vert A_1\vert=3$.
\end{subcase}

\begin{subsubcase}\label{130420a}
At least two reactions in $A_2$ are $1$-resourced.
\end{subsubcase}

Due to Lemma~\ref{020520a}, we may suppose $\{(\emptyset, 4,1), (\emptyset,1,2)\}\subseteq A_2$ 
or $\{(\emptyset, 4,1), (\emptyset,2,3)\}\subseteq A_2$.
First, we deal with the case $\{(\emptyset, 4,1), (\emptyset,1,2)\}\subseteq A_2$.
Note that  $f(1)$, $f(2)$, and $f(4)$ cannot be exhausted by $A_1$
and thus they have to be exhausted by the only reaction in $A_3$ together with the remaining reaction in $A_2$.
Hence, it follows that $(\emptyset, 23, 34)\in A_2$ and either $(2,3,24)\in A_3$ or $(2,34,124)\in A_3$.
Furthermore, $(3, \emptyset, 3)\in A_1$. However, it can be verified that the remaining two reactions in $A_1$ cannot exhaust all $f(X)$ for $X=123, 124, 134, 234$. The case $\{(\emptyset, 4,1), (\emptyset,2,3)\}\subseteq A_2$ is similar and simpler as now all $f(X)$ for $X=1,2,3,4$ cannot be exhausted by $A_1$.

\begin{subsubcase}\label{200420a}
Exactly one reaction in $A_2$ is $1$-resourced. 
\end{subsubcase}

We may suppose $(\emptyset, 4,1)\in A_2$ due to Lemma~\ref{020520a}.
If more than one of the reactions in $A_1$ is $1$-resourced, then we are reduced to Case~\ref{130420a} by means of Lemma~\ref{180420b}. Else if none of the reactions in $A_1$ is $1$-resourced, 
by considering $f(X)$ for $X=1,2,3,4$, it leads to one of the following two possibilities:
\begin{enumerate}
\item $A_2= \{ (\emptyset,4,1 ),(\emptyset,23, 34),(\emptyset, 12, 23)\}$ while $(2, 3, 24)\in A_3$ or $(2,34,124)\in A_3$;
\item $A_2 = \{(\emptyset,4,1), (\emptyset,23, 34), (\emptyset, 13, 24) \}$ while $(3, 1, 23)\in A_3$ or $(3,14,123)\in A_3$.
\end{enumerate}
However, in either case, four of the $f(X)$ for $X=12,13,14,23,24,34$ are not completely accounted by $A_2\cup A_3$, but the three reactions in $A_1$ cannot exhaust them all. Therefore, we may suppose exactly one of the reactions in $A_1$ is $1$-resourced.

First, we suppose $(4, \emptyset, 4)\in A_1$. Since none of the $f(X)$ for $X=12,13,14,$ $23, 24,34$ is exhausted, it follows that the only reaction in $A_3$ must exhaust two of them and each of the remaining four reactions in $A_1\cup A_2$ must exhaust one of them. By considering $f(12)$, it implies that $(2,3,24)\in A_3$.
However, the remaining parts of $f(14)$ cannot be exhausted by any of the remaining reactions in $A_1\cup A_2$.


Next, we suppose $(1, \emptyset,1)\in A_1$ or $(2, \emptyset, 2)\in A_1$. Similarly, the only reaction in $A_3$ must exhaust two among $f(X)$ for $X=12,13,14,23,24,34$ and each of the remaining reactions in $A_1\cup A_2$ must exhaust one of them. However, $f(34)=234$ is yet to be accounted for and thus cannot be exhausted by any one of them.

Finally, we suppose $(3, \emptyset, 3)\in A_1$. First, if $(2,3,24)\in A_3$, then it can be verified that the two remaining reactions in $A_1$ cannot exhaust all $f(X)$ for $X=123,124,134,234$. Next, if $(4,1,24)\in A_3$, then the two remaining reactions in $A_2$ cannot exhaust all $f(X)$ for $X=1,2,3,4$.
Otherwise, if the reaction in $A_3$ is other than the two above, then it should now be easier to verify that a similar contradiction can be reached.

\begin{subsubcase}
None of the reactions in $A_2$ are $1$-resourced.
\end{subsubcase}

Due to Lemma~\ref{180420b}, we may assume that none of the reactions in $A_1$ are $1$-resourced as well for otherwise, we are reduced to Cases~\ref{130420a} and \ref{200420a}. However, by some simple counting, it follows that
each reaction in $A_1\cup A_2$ accounts for two of the elements 
of $f(X)$ for $X=12,13,14,23,24,34$ while the only reaction in $A_3$ accounts for the remaining four of the $16$ elements. However, this is impossible.
\end{proof}

The following corollary now follows from Theorem~\ref{200420b}.

\begin{corollary}
	For $\vert S\vert=4$, the largest reaction system rank attainable by an $rs$ function specified by a strictly minimal reaction systems over $S$ is eight.
\end{corollary}

\section{Conclusion}

As a conclusion, our study shows that even strictly minimal reaction systems present some interesting and challenging purely mathematical problems. Precisely, our work manages to answer Question~\ref{2003a} only up to the quaternary alphabet and provides further evidence that reaction system rank is combinatorially hard to be dealt with rigorously.

As opposed to the case of $\mathcal{A}_{(1\,2\,3)}$, our proof of Theorem~\ref{200420b} shows that
the function specified by $\mathcal{A}_{(1\,2\,3\,4)}$ is complicated in the sense that no set of less than eight reactions  is sufficient to specify it. Hence, we expect higher alphabets to behave similarly.

\begin{conjecture}\label{210520a}
	For $\vert S\vert \geq 4$, the reaction system rank of any $rs$ function over $S$ specified by a strictly minimal reaction system induced by a cycle of length $\vert S\vert$ is $2 \vert S\vert$. 
\end{conjecture}

For the quinternary and the next few alphabets, we can utilize the case reduction strategy used in our proof of Theorem~\ref{200420b} to improve the efficiency of any computational solution. However, this would certainly meet its limitation. Therefore,   
special functions specified by strictly minimal reaction system could possibly be ingeniously identified for which their reaction system rank can be feasibly shown to be $2\vert S\vert$, thus answering Question~\ref{2003a} for all alphabets without actually proving Conjecture~\ref{210520a}. 

Finally, various decision problems regarding biological properties or dynamical behavior of reaction systems have been established to be NP-complete or PSPACE-complete, such as in \cite{azimi2015complexity, dennunzio2019complexity}. Therefore, as a potential future direction inspired by this work, one can study the complexity of problems related to irreducibility and reaction system rank for the class of minimal or strictly minimal reaction systems.

\section*{Ackowledgment}

The first author acknowledges support of Fundamental Research Grant Scheme \linebreak No.~203.PMATHS.6711644 of Ministry of Education, Malaysia, and Universiti Sains Malaysia.


\end{document}